\documentclass[3p,preprint,12pt]{elsarticle}

\usepackage{amssymb}
\usepackage[numbers]{natbib}
\usepackage[utf8]{inputenc}
\usepackage[english]{babel}
\usepackage{amsthm}
\usepackage{graphicx}
\usepackage{float}
\usepackage{caption}
\usepackage[font=small,labelfont=bf]{caption}

\usepackage{color}

\usepackage{hyperref}
\usepackage{amsmath}
\usepackage{amssymb}
\usepackage{amscd}
\usepackage{amssymb}
\usepackage{amsfonts}
\usepackage{amsbsy}

\usepackage{multicol}
\newtheorem{theorem}{Theorem}
\newtheorem{definition}{Definition}
\newtheorem{remark}{Remark}
\newtheorem{lemma}{Lemma}
\newtheorem{proposition}{Proposition}

\newtheorem*{fundlemma}{Fundamental Lemma}
\newtheorem*{example}{Example 1}
\newtheorem*{example1}{Example 2 (Relay Systems)}
\newcommand{\R}{\mathbb{R}}

\newcommand{\N}{\mathbb{N}}
\newcommand{\s}{\Sigma}
\newcommand{\g}{\Gamma}
\newcommand{\M}{M}

\journal{Journal de Mathématiques Pures et Appliquées}

\begin{document}

\begin{frontmatter}

 \title{Limit sets of discontinuous vector fields on two-dimensional manifolds}
\author{Rodrigo D. Euz\'ebio\fnref{1}}
\ead{euzebio@ufg.br}
\fntext[1]{corresponding author}
\author{Joaby de S. Jucá}
\ead{joabyjuca.school@gmail.com}
\address{Departamento de Matemática, IME-UFG, R. Jacarandá,\\
	Campus Samambaia, Zip Code 74001-970, Goiânia, GO, Brazil}

\begin{abstract}
In this paper the asymptotic behavior of trajectories of discontinuous vector fields is studied. The vector fields are defined on a two-dimensional Riemannian manifold $M$ and the confinement of trajectories on some suitable compact set $K$ of $M$ is assumed. The behavior of the global trajectories is fully analyzed and their limit sets are classified. The presence of limit sets having non-empty interior is observed. Moreover, the existence of the so called sliding motion is allowed on $M$. The results contemplate a list of possible limit sets as well the existence of non-recurrent dynamics and the presence of nondeterministic chaos. Some examples and classes of systems fitting the hypotheses of the main theorems are also provided in the paper.
\end{abstract}
\end{frontmatter}

\begin{multicols}{2}
\section{Introduction}

The theory of discontinuous vector fields (DVF for short) addresses the study of trajectories which eventually looses smoothness by reaching so\-me boundary region of the phase portrait. Generally that boundary is taken as a co-dimension one manifold therefore DVF have a close connection to smooth vector fields defined on manifolds with boundaries. Although the last can somehow be seen as a particular case of the former, several results dealing with the boundary contact have inspired and motivated the development of DVF. More specifically, this new theory have been widely studied throughout the last three decades thanks to the pioneering work of A. F. Filippov \cite{Filippov88}, which developed a schematic study of DVF. It have received special attention also due to a strong connection with applications in areas as electronics, mechanics, control theory, economy, biology, medicine, among others (see for instance \cite{diBernardo08,ThulCoombes10,TonnelierGerstner03,Utkin08} and references therein).

The dynamic of a DVF is generally more complicated than a smooth vector field because there are several ingredients of the vector fields playing some role. In fact, to fix ideas, even in the simplest case of a DVF formed by two vector field whose trajectories are separated by a common frontier, one must consider not only the dynamic of each particular vector field but also their interaction to the single common boundary. Moreover, the trajectories of such vector fields eventually slide on the boundary by an amalgamation processes occurring at the moment of the collision which generates new trajectories. For that reason the first steps toward the construction of a consistent theory of DVF require the establishing and validation of new results analogous to classical ones. In this direction we highlight some of those results already established into the discontinuous context in the areas of
Stability \cite{BrouckePughSimic01},
Chaos \cite{Buzzi-Carvalho-Euzebio18,Colombo2011}, 
Bifurcation \cite{GuardiaSearaTeixeira11,KuznetsovRinaldiGragnani03},
Closing Lemma \cite{DeCarvalho2016} and the Peixoto's Theorem \cite{GomideTeixeira}. 

A particular gap to a whole comprehension of DVF is the lack of global results. Effectively in many cases local aspects must be assumed so local results as the submersion theorem simplify coordinates and consequently calculations are generally more treatable. In our context, however, no local aspects are required but only contact conditions on a compact portion $K$ of the phase portrait, that is, the approach is {\it semi-local}. More specifically we consider DVF defined on a two-dimensional Riemannian manifold $\M$ which is locally split into some regions, each region being delimited by the connected components of a regular co-dimension one manifold $\s$. Some conditions are imposed to the compact $K$. For instance, we shall assume that it contains finite critical elements and intersects only one connected component of $\s$. We also assume that $K$ is positive invariant for some trajectory in the sense of the Poincaré-Bendixson Theorem for smooth vector fields.


The goal of this paper is to study asymptotic aspects of a {\it maximal} trajectory of a DVF. More precisely we are interested in obtaining the limit sets of such trajectories. The main result of the paper provides a fine classification of those objects for the class of DVF we are dealing with. We highlight that the obtained limit sets may present chaotic behavior. A preliminary study of limit sets can be found for the plane in \cite{Buzzi-Carvalho-Euzebio18}. In such paper the authors allow discontinuities but they do not consider sliding motion, obtaining then the non-generic situation where tangency points coincide.


This paper is organized as follows. In Section \ref{secao_preliminares} we present the fundamental notions of DVF (Subsection \ref{PSVF}) and we discuss some extensions of discontinuous vector fields for transitions between sliding and escaping regions (Subsection \ref{extension}). In Section \ref{heart} we pre\-sent the main results and a brief discussion of them. In Section \ref{global-analysis} we analyze the behavior of a maximal trajectory contained on a compact set $K$. We also establish auxiliaries results concerning pseudo-cycles (Subsection \ref{pseudo-cycles}), mild pseudo-cycles and chaotic sets (Subsections \ref{mild-pc_and_chaotics} and \ref{chaotic}) and pseudo-graphs (Subsection \ref{chaotic}). In Sections \ref{proof} we prove the main results of the paper and Section \ref{linear} provides some features of discontinuous linear vector fields. Finally in Section \ref{final} we present some conclusions on the main achievements of the paper.


\section{Preliminaries}\label{secao_preliminares}

\subsection{Discontinuous vector fields}\label{PSVF}

Let $M$ be a smooth two-dimensional Riemannian manifold. From Nash Embedding Theorem we can assume that $M$ is isometrically embedded into some Euclidean space $\R^n$. Suppose that $M$ admits a smooth function $f:M\to\R$ having $0\in\R$ as regular value in such way that $\Sigma = f^{-1}(0)$ 
splits $M$ into two connected disjoint regions 
$\Sigma^-=\{p\in M; f(p)< 0)\}$ and $\Sigma^+=\{p\in M; f(p)> 0)\}$. We call $\s$ the switching manifold and we consider the discontinuous vector fields $Z=(X,Y)$ on $M$ defined by 
\begin{equation}\label{general_system}
Z(p)=\left\{\begin{array}{lll}
X(p),\ \ \ \mbox{if } \ f(p)\geq 0,\\
Y(p),\ \ \ \mbox{if } \ f(p)\leq 0,
\end{array}
\right.
\end{equation}
where $X$ and $Y$ are smooth vector fields defined on  $M$. As we will see the number of regions of $M$ split by $\s$ may be higher than two but it is sufficient for our purposes (cf. Remark \ref{zones-number}).\\
Now, in order to study the behavior of trajectories of DVF on $\Sigma$ we must introduce some notations.
Indeed, given a vector field $X$ on $M$ and a point $p\in\s$ we consider the Lie derivatives $X. f(p) = \langle X(p),\nabla f(p)\rangle$ and $X^i. f(p)=\langle \nabla X^{i-1}. f(p),X(p)\rangle,\ i\geq 2$, where $\langle\cdot , \cdot\rangle$ is the canonical inner product on $\R^n$ on which $M$ is embedded.

On $\Sigma$ we distinguish three regions satisfying $(X.f(p))\cdot (Y.f(p))\neq 0$, that is, transversely to $\s$.
i) The \textit{sewing region} $\Sigma^c$ formed by the points $p\in\Sigma$ such that the trajectory of $X$ (resp. $Y$) meet $p$ in finite future time and the trajectory of $Y$ (resp. $X$) meets $p$ in finite past time. ii) The \textit{escaping region} $\Sigma^e$ formed by the points $p\in\Sigma$ such that the trajectories of $X$ and $Y$ meets $p$ in finite past time. Lastly iii) the \textit{sliding region} $\Sigma^s$ formed by the points $p\in\Sigma$ such that the trajectories of $X$ and $Y$ meets $p$ in finite future time.

The points $p\in\Sigma$ such that $X. f(p) = 0$ (resp. $Y. f(p) = 0$) are called tangency points of $X$ (resp. $Y$). They are denoted by $\s^t$.
We say that a tangency point $p$ has contact of order $n\in\N$ if the Lie derivatives $X^k.f(p)$ vanish for $k<n$ and 
$X^n.f(p)\neq 0$. We also classify tangency points according to the following: let $X$ and $Y$ be two smooth vector fields on $M$ and suppose that $\nabla f(p)$ points to the interior of the region $\s^+$ where $Z$ coincides with $X$. We say that $p\in\Sigma$ is an {\it invisible tangency} if the contact order $r$ of a trajectory of $X$ (resp. $Y$) passing through $p$ is even and $X^{r}.f(p)<0$ (resp. $Y^{r}.f(p)>0$). On the other hand, we say that $p\in\Sigma$ is a {\it visible tangency} if the contact order $r$ of a trajectory of $X$ (resp. $Y$) passing through $p$ is even and $X^{r}.f(p)>0$ (resp. $Y^{r}.f(p)<0$) or it is a tangent point with odd contact for $X$ (resp. $Y$).

For our purposes it is also interesting to consider a special configuration of tangency point, namely the case where the vector fields share such a point. We call that points {\it double tangency points}. For a DVF $Z=(X,Y)$ we say that a double tangency is \textit{elliptical} if it is an invisible tangency for both $X$ and $Y$, \textit{parabolic} if it is a visible tangency of even order for a vector field and invisible one for the other and \textit{hyperbolic} if the double tangency is visible of even order for both $X$ and $Y$. In this paper we also refer to elliptical tangency points between two sewing regions by tangency point of type I (see Figure \ref{fig1-11}). On the other hand, let $p\in\Sigma$ an invisible tangency for $X$. If (i) $p$ is on the boundary of a sliding region and attracts sliding orbits and (ii) $p$ is a tangency of odd order for $Y$ then we call $p$ a tangency of type II (see Figure \ref{fig1-12}). We notice that if $p$ satisfies the previous conditions except (ii) and it is regular for $Y$ then $p$ repels sliding orbits.

\includegraphics[scale=1.1]{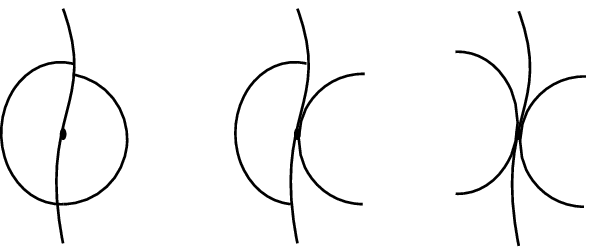}
\captionof{figure}{Types of double tangency points: elliptical between sewing regions (or tangency of type I), parabolic and hyperbolic, respectively.}\label{fig1-11}

\begin{center}
\includegraphics[scale=1.1]{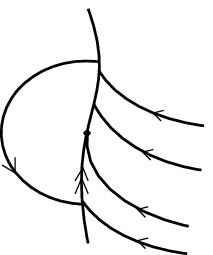}
\captionof{figure}{A tangency point of type II. In this case the double tangency point necessarily have an odd order contact.}\label{fig1-12}
\end{center}

In order to deal with sliding motion in the DVF we adopted the Filippov convention according to \cite{Filippov88} to define a vector field $Z^\Sigma$ on $\Sigma^e\cup\Sigma^s$ such that $Z^\s(p)$ is the vector on $T_pM$ tangent to $\s$ obtained by a convex combination of $X(p)$ and $Y(p)$. That is,
\begin{equation}\label{Filippov_vector-field}
Z^\Sigma(p)=\frac{Xf(p).Y(p)-Yf(p).X(p)}{Xf(p)-Yf(p)}.
\end{equation}
with $p\in \Sigma^e\cup\Sigma^s$.

In this paper we refer to $Z^\Sigma$ as {\it Filippov vector field}. The points $p\in\Sigma^e\cup\Sigma^s$ such that $Z^\Sigma(p)=0$, that is, $p\in\s$ such that $X(p)$ and $Y(p)$ are collinear and points to opposite direction are called \textit{pseudo equilibrium} of $Z$. The trajectories of $Z^\Sigma$ may constitute a trajectory of $Z$ according to the following definition. We alert that other convention could lead to different results, see for instance \cite{diBernardo08,BrouckePughSimic01}.

\begin{definition}\normalfont
A \textit{global trajectory} $\Gamma_Z(t,p)$ of a discontinuous vector field $Z$ is the trace of a
continuous curve obtained by concatenation of 
trajectories of $X$ and/or $Y$ and/or $Z^\Sigma$. A \textit{maximal trajectory} $\g_Z=\Gamma_Z(t,p)$ is a global trajectory that
cannot be extended by any concatenation of trajectories of $X,\ Y$ or $Z^\Sigma$. In this case, we call $I=(\tau^-(p),\tau^+(p))$ the \textit{maximal interval of the solution} $\g_Z$. The \textit{positive maximal trajectory} $\g_Z^+(t,p)$ is the portion of a maximal trajectory of $\g_Z(t,p)$ for which $t>0$.
\end{definition}

Notice that it is possible that we have $\tau^+(p)=+\infty$ for a maximal trajectory $\g_Z(t,p)$ such that $\g_Z(t,p)\neq q$ for $t<t_0$ and $\g_Z(t,p)=q$ for all $t\geq t_0>0$. (See the tangent point of type II in Figure \ref{fig1-12}).

\begin{definition}\normalfont\label{def-omega-limite}
Given $\Gamma_Z(t,p)$ a maximal trajectory of $Z$, the set 
$$
\begin{array}{c}
\omega(\Gamma_Z(t,p)) =
\left\{q\in M:\  \exists\ (t_n)\subset\R\ \right.\\
 \mbox{with} \displaystyle\lim_{t_n\to\tau^+(p)}\Gamma_Z(t_n,p)= q\}
\end{array}
$$
is called \textit{$\omega$-limit set of} $\Gamma_Z(t,p)$.
\end{definition}

A more detailed presentation of trajectories of DVF can be found in \cite{GuardiaSearaTeixeira11}.

\begin{definition}\label{def_pseudo-cycle}\normalfont
Let $Z$ be a discontinuous vector field. Consider $\Gamma$ a closed maximal trajectory of $Z$ such that $\Gamma\cap\Sigma\neq\emptyset$ and assume it contains neither equilibria nor pseudo-equilibria. Then, we say that $\Gamma$ is a \textit{pseudo-cycle} if
\begin{itemize}
\item[(i)] $\Gamma$ is positively or negatively invariant and
\item[(ii)] $\Gamma$ does not contain a proper maximal trajectory.
\end{itemize}
\end{definition}

The last definition is a refinement of the concept of pseudo cycle usually considered in the literature. That allow us to explore the pseudo cycles in a more accurate way. In particular we distinguish three types of pseudo-cycles: ($a$) \textit{crossing} pseudo-cycles, which are those satisfying $\Gamma\cap\Sigma\subset\Sigma^c$, see Figure \ref{fig1-15} ($a$), ($b$) \textit{tangent} pseudo-cycles, when $\Gamma\cap\Sigma^t\neq\emptyset$ and $\Gamma$ does not present sliding motion, see Figure \ref{fig1-15} ($b1-2$), and ($c$) \textit{sliding} pseudo-cycles, which are the pseudo-cycles having sliding motion, see Figure \ref{fig1-15} ($c1-3$).

\begin{center} 
\includegraphics[scale=0.6]{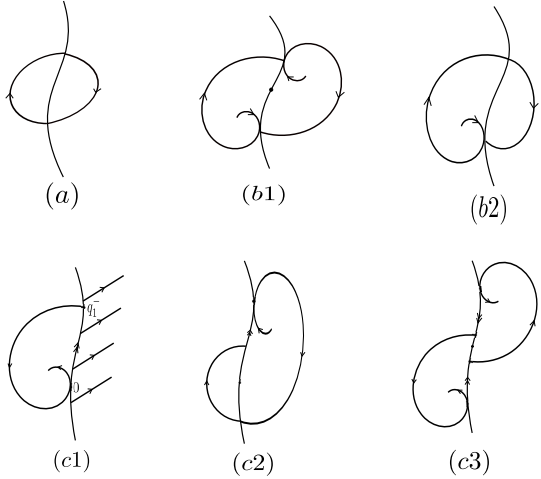}
\captionof{figure}{Examples of crossing ($a$), tangent ($b1-2$) and sliding ($c1-3$) pseudo-cycles.}\label{fig1-15}
\end{center}

\begin{definition}\label{mild_pseudo-cycle}\normalfont
A maximal trajectory as introduced in Definition \ref{def_pseudo-cycle} which do not satisfy one or two of the itens (i) and (ii) is called a \textit{mild pseudo-cycle}.
\end{definition}
In particular we distinguish three types of mild pseudo-cycle: (a) \textit{mild pseudo-cycle of type I} when only condition (i) fails, (b) \textit{mild pseudo-cycle of type II} when only condition (ii) fails, and (c) \textit{mild pseudo-cycle of type III} when both conditions (i) and (ii) fail (see Figure \ref{fig1-17} for some examples).

\begin{center} 
\includegraphics[scale=0.8]{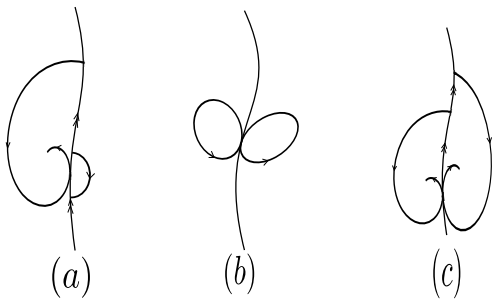}
\captionof{figure}{Examples of mild pseudo-cycles of type I (a), type II (b) and type III (c).}\label{fig1-17}
\end{center}

We shall see in Subsection \ref{mild-pc_and_chaotics} that some mild pseudo-cycles may present chaotic behavior so next we adapt that concept to DVF. Accordingly, one of the well accepted definitions of chaos in the literature is the one assuming topological transitivity, sensitive dependence on initial conditions and density of periodic orbits. These properties concerns Devaney's conditions for the existence of chaos for smooth systems.\\
Now let $\Lambda\subset M$ be a compact invariant set for the DVF $Z$. We say that $Z$ is \textit{topologically transitive} on $\Lambda$ if for any pair of non-empty, open sets $U$ and $V$ in $\Lambda$, there exist $p\in U,\ \Gamma^+_Z(t,p)$ a positive maximal trajectory and $t_0 > 0$ such that $\Gamma_Z(t_0,p)\in V$. Moreover we say that $Z$ exhibits \textit{sensitive dependence} on $\Lambda$ if there exist a fixed $r>0$ satisfying $r< \mbox{diam}(\Lambda)$ such that for each $x\in\Lambda$ and $\epsilon>0$ there exist $y\in B_\epsilon(x)\cap\Lambda$ and positive maximal trajectories $\Gamma_Z^+(t,x)$ and $\Gamma_Z^+(t,y)$ on $\Lambda$ passing through $x$ and $y$, respectively, satisfying $d(\Gamma_Z^+(t,x),\Gamma_Z^+(t,y))>r$ for some $t>0$, where $d$ is the Euclidian distance in $\mathbb{R}^n$ on which the Riemannian manifold $M$ have been isometrically embedded.

\begin{definition}\normalfont
We say that a compact invariant set $\Lambda$ is \textit{chaotic} for a discontinuous vector field $Z$ if 
\begin{enumerate}
\item[a)] $Z$ is topologically transitive on $\Lambda$;
\item[b)] $Z$ exhibits sensitive dependence on $\Lambda$ and
\item[c)] the periodic orbits of $Z$ are dense on $\Lambda$.
\end{enumerate}

Moreover, we distinguish three situations: (i) $\Lambda$ is a \textit{chaotic set of type I} if it has empty interior and does not present sliding motion; (ii) $\Lambda$ is a \textit{chaotic set of type II} if it has empty interior presenting sliding motion and (iii) $\Lambda$ is a \textit{chaotic set of type III} if it has non-empty interior. 
\end{definition}

Next we define pseudo-graphs. These objects are an important part of Theorem \ref{TPB_Euzebio-Juca} stated in the next section.

\begin{definition}\normalfont
Let $Z$ be a discontinuous vector field. A closed curve $\Gamma$ is a \textit{pseudo-graph} if $\Gamma\cap\Sigma\neq\emptyset$ and it is an union of trajectory-arcs of $Z$ joining equilibrium or pseudo-equilibrium (see Figure \ref{fig1-18}).
\end{definition}

\begin{center} 
\includegraphics[scale=1]{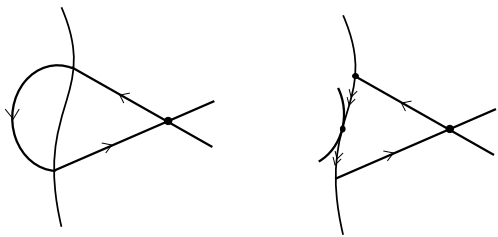}
\captionof{figure}{Examples of pseudo-graphs.}\label{fig1-18}
\end{center}


\subsection{Extension of Filippov vector fields beyond the boundary of $\s^{s,e}$}\label{extension}

\medskip


In what follows we briefly discuss the extension of one-dimensional vector fields to the boundaries of their domains in order to allow infinitely many transitions over $\Sigma$. Indeed, consider $Z$ as in equation (\ref{general_system}) and $Z^\Sigma$ the associated Filippov vector field defined on $\s^{s,e}$. We shall extend the Filippov vector field to the closure of escaping and sliding regions, that is, to their adjacent tangency points if it is possible. Assume that $p$ is the common boundary between $\Sigma^s$ and $\Sigma^e$ and consider
$$
L^{e,s} = \lim_{q\to p}Z^\Sigma(q).
$$
The limit is consider at $q\in\Sigma^e$ for $L^e$ and $q\in\Sigma^s$ for $L^s$. We split the analysis into two cases:
\begin{itemize}
	\item[(a)] $L^e=L^s$. In this case we define
	$$
	L=\lim_{q\to p} Z^\Sigma(q)=L^e=L^s,
	$$
	and therefore $Z^\Sigma(p)=L$. If $L\neq 0$ then the Filippov vector field in a neighborhood of $p$ points  to the same direction than $L^{e,s}$ and therefore the trajectory flows through p from $\Sigma^e$ to $\Sigma^s$ or vice-versa. In other words, $p$ is a regular point for the extended Filippov vector field. If $L=0$ then $p$ is a pseudo-equilibrium for it.
	\item[(b)] If $L^e\neq L^s$ both being non-zero pointing to the same direction, then by a re-scaling of time of  	$Z^{\Sigma^e}$ or $Z^{\Sigma^s}$ we obtain the same situation of (a) and we are done. If the vector fields $Z^{\Sigma^e}$ and $Z^{\Sigma^s}$ point to opposite directions, then we cannot extend the 
	Filippov vector field to $p$ because we obtain an one-dimensional Filippov dynamics around an attractor or repulsor point. The same happens when $L^e = 0$ 
	or $L^s = 0$, i.e., the Filippov vector field cannot be extended beyond its boundary.
\end{itemize}

\begin{remark}\label{equilibrium}\normalfont
	Notice that using the expression of the Filippov vector field given in (\ref{Filippov_vector-field}), when $p\in\Sigma$ is an equilibrium point of $Y$ or $X$ then $\displaystyle\lim_{q\to p}Z^\Sigma(q)=0$ and therefore $p$ is an equilibrium point of the extended Filippov vector field.
\end{remark}

\section{Main results}\label{heart}

\subsection{Statement of the main results}

In this paper we develop a semi-local study of DVF through an asymptotic analysis of trajectories in order to study limit sets and their properties. First we assume that trajectories transit from and to the switching manifold $\Sigma$ until they remain outside $\Sigma$ or inside it. After that we assume that infinitely many transitions on $\Sigma$ occur. Finally, we explore nontrivial recurrences of the chaotic sets that we obtained as limit sets.

In what follows we consider a discontinuous vector field $Z=(X,Y)$ on a two-dimensional Riemannian manifold $M$ and $K\subset M$ a compact set satisfying some hypotheses on $K$ and $Z$. Most ot them are only adapted from Poincaré-Bendixson Theorem to the discontinuous context.
\begin{itemize}
	\item[($K_1$)] $K$ is contained on some coordinate neighborhood of $M$;
	\item[($K_2$)] $\s\cap K$ is a smooth curve splitting $K$ into two disjoint connected components, $K\cap\s^+$ and $K\cap\s^-$;
	\item[($Z_1$)] $X$ and $Y$ have a finite number of equilibrium points on $K$;
	\item[($Z_2$)] $Z$ has only isolated pseudo-equilibrium points on $K\cap\s$;
	\item[($Z_3$)] both $X$ and $Y$ have at most one tangent point on $K\cap\s$.
\end{itemize}
\begin{remark}\label{zones-number}
	The hypothesis ($K_1$) is the motivation for referring to the approach of the paper as semi-local. However it may not be quite accurate because some global scenarios may be referred as semi-local. We point, for instance, to the whole Cartesian plane $\R^2$ or the two-dimensional sphere minus one point, among others. We also remark that hypothesis ($K_2$) motivates the two zones presented in equation \eqref{general_system} but one can define such it for any number of zones. 
\end{remark}

We now establish the main results of the paper.

\begin{theorem}\label{TPB_Euzebio-Juca}
	Consider a DVF $Z=(X,Y)$ on a two-dimensional Riemannian manifold $M$ and let  $K\subset M$ be a compact subset of $M$.  Assume that the hypotheses $(K_i)_{i=1}^2$ and $(Z_j)_{j=1}^3$  are fulfilled and that $Z$ has a positive maximal trajectory $\Gamma^+_Z(t,p)$ contained on $K$. Then the $\omega$-limit set of $\Gamma_Z(t,p)$ is one of the following objects:
	\begin{itemize}
		\item[(i)] an equilibrium of $X$ or $Y$;
		\item[(ii)] a periodic orbit of $X$ or $Y$;
		\item[(iii)] a graph of $X$ or $Y$;
		\item[(iv)] a pseudo-equilibrium of $Z$;
		\item[(v)] a (crossing, tangent or sliding) pseudo-cycle of $Z$;
		\item[(vi)] a mild pseudo-cycle of type I, II or III of $Z$;
		\item[(vii)] a pseudo-graph of $Z$;
		\item[(viii)] a tangency point of type I or II;
		\item[(ix)] a chaotic set of type III.
	\end{itemize}
\end{theorem}

We highlight some points concerning Theorem \ref{TPB_Euzebio-Juca}. First, the sliding pseudo-cycles may have distinct topological types. For instance, some of the sliding pseudo-cycles we obtain are contained on $\Sigma^+$ (or $\Sigma^-$) and possess only a segment of slide. However, we also obtain sliding limit cycles occupying both $\Sigma^+$ and $\Sigma^-$ as well as occupying two disjoint sliding regions. We also notice that in statement $(ix)$ of Theorem \ref{TPB_Euzebio-Juca} it may occur a $\omega$-limit set having nonempty interior and other interesting topological properties. The next result describes such features.

\begin{theorem}\label{proprerties_of_lambda}
	Let $Z$ be a DVF such as in Theorem \ref{TPB_Euzebio-Juca} and assume that $Z$ has a parabolic or hyperbolic tangency point $p_0\in\partial(\Sigma^e\cup\Sigma^s)$ on the` interior of $K$. Assume that there exist a strictly increase sequence of times $t_i>0$ such that $\g_Z(t_i,p)\in\s^s$ and $\g_Z(t_{i+1},p)\in\s^e$. Then $Z$ admits a chaotic set $\Lambda$ of type III as $\omega$-limit of a maximal trajectory in such way that every point $q$ on its interior satisfies the following statements:
	\begin{itemize}
		\item[(a)] there exist a periodic orbit of $Z$ on $\Lambda$ passing through $q$;
		\item[(b)] there exist a dense orbit of $Z$ on $\Lambda$ passing through $q$;
		\item[(c)] $q$ is a non-trivial recurrent point on $\Lambda$;
		\item[(d)] $q$ is a non-trivial non-wandering point on $\Lambda$.
	\end{itemize}
\end{theorem}

In order to prove Theorems \ref{TPB_Euzebio-Juca} and \ref{proprerties_of_lambda} it is important to state the following result which is a generalization of Theorem 1 of \cite{Buzzi-Carvalho-Euzebio18}.

\begin{fundlemma}
	Consider a DVF $Z=(X,Y)$ on a two-dimensional Riemannian manifold $M$ and let  $K\subset M$ be a compact subset of $M$.  Assume that the hypotheses $(K_i)_{i=1}^2$ and ($Z_1$) are fulfilled and $Z$ has a positive maximal trajectory $\Gamma^+_Z(t,p)$ contained on $K$. If there exist $t_0>0$ such that $\Gamma_Z(t,p)\notin \overline{\Sigma^e\cup\Sigma^s},\ \forall t>t_0$, then the $\omega$-limit set of $\Gamma_Z(t,p)$ is one of the following objects:
	\begin{itemize}
		\item[(i)] an equilibrium of $X$ or $Y$;
		\item[(ii)] a periodic orbit of $X$ or $Y$;
		\item[(iii)] a graph of $X$ or $Y$;
		\item[(iv)] a crossing pseudo-cycle of $Z$;
		\item[(v)] a mild pseudo-cycle of type II of $Z$;
		\item[(vi)] a pseudo-graph of $Z$;
		\item[(vii)] a tangency point of type II.
	\end{itemize}
	On the other hand, if $\Gamma_Z(t,p)\in\Sigma^e\cup\Sigma^s$ for $t$ sufficiently large then $\omega(\Gamma_Z(t,p))$ is
	one of the following objects:
	\begin{itemize}
		\item[(viii)] a pseudo-equilibrium of $Z$ or
		\item[(ix)] a tangency point of type II.
	\end{itemize}
\end{fundlemma}

\begin{remark}\label{Remark}\normalfont
	In order to prove Theorems \ref{TPB_Euzebio-Juca} and \ref{proprerties_of_lambda} we will consider a positive maximal trajectory $\Gamma_Z^+(t,p)$ contained on $K$ in such way that there exists $s_n\to \tau^+(p)$ and $t_n\to\tau^+(p)$ satisfying $\Gamma_Z(s_n,p)\in\overline{\Sigma^e\cup\Sigma^s}$ and $\Gamma_Z(t_n,p)\notin\Sigma$. In other words, we will assume that $\Gamma_Z(t,p)$ visits and leaves $\overline{\Sigma^e\cup\Sigma^s}$ infinitely many times. Otherwise we may apply the Fundamental Lemma straightforwardly. Moreover, we notice that the classical planar Poincaré-Bendixson Theorem applies for smooth vector fields $X$ on compact sets contained on a coordinate neighborhood of $M$. Indeed its expression induced by a chart on some open neighborhood of $\R^2$ is a smooth vector field topologically equivalent to $X$.
\end{remark}

\subsection{Brief discussions}

The results presented in this section aims to extend the study of limit sets for discontinuous vector fields defined on manifolds under the presence of sliding motion. In a generic context the existence of sliding is mandatory because otherwise a coincidence of tangency points must take place (see for instance the work on structural stability in \cite{Filippov88}). Once sliding implies some non-determinism in trajectories, a chaotic regime may appear. This situation is particularly observed when some arc of trajectory connects both sliding and escape as stated in Theorem \ref{proprerties_of_lambda}. In particular, some non-trivial minimal sets can emerge as reported in \cite{Buzzi-Carvalho-Euzebio16}. As far as the authors know, a better understanding of those sets is far from be achieved.

Although hypothesis preceding Theorem \ref{TPB_Euzebio-Juca} are mainly due to Poincaré-Bendixson Theorem as commented before, they also avoid the presence of non-trivial minimal sets, we refer to $(K_2)$ and $(Z_3)$. We illustrate the meaning of them in the following two examples. We start by highlighting hypothesis $(K_2)$.

\begin{example}\label{exemplo-3-zonas}
	Consider the following planar linear DVF with three linearity zones
	\begin{equation}\label{system3zones}
	Z(x,y)=\left\{\begin{array}{l}
	(-y-1,x),\;\mbox{if} \|(x,y)\|\geq 1,\\
	(-2y,x),\;\mbox{if} \|(x,y)\|\leq 1.
	\end{array}
	\right.
	\end{equation}
	
	\begin{center}[h]
		\includegraphics[scale=0.65]{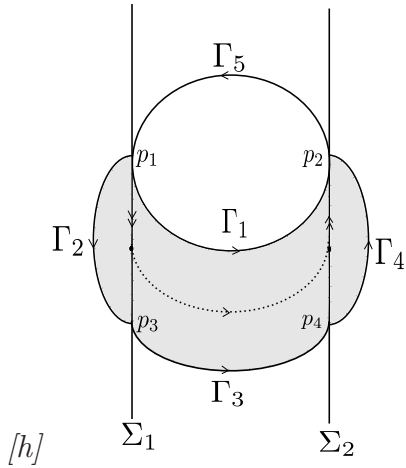}
		\captionof{figure}{A non-trivial minimal set of a planar discontinuous linear vector field with three linearity zones.}\label{fig1-7}
	\end{center}

	There exists two discontinuity lines, namely $\Sigma_1=\{-1\}\times\R\subset\R^2$ and $\Sigma_2=\{1\}\times\R\subset\R^2$, we call $\Sigma=\Sigma_1\cup\Sigma_2$. Moreover, it is easy to see that $\Sigma$ can be split into intervals of the form $\Sigma_1^c=\{-1\}\times[(-\infty,-1)\cup(0,+\infty)]$, $\Sigma_1^e=\{-1\}\times(-1,0)$, $\Sigma_2^c=\{1\}\times[(-\infty,-1)\cup(0,+\infty)]$ and $\Sigma_2^s=\{1\}\times(-1,0)$. The Filippov vector field
	on $\Sigma_1^e$ and $\Sigma_2^s$ is then $Z^\Sigma(x,y)=(0,x)$.
	
	Let $\mathit{\Lambda}$ be the set on $\R^2$ delimited by the curves $\Gamma_i$, $i=1\ldots,5$,
	where $\Gamma_1,\ \Gamma_3$ and $\Gamma_5$ are arc of trajectories of $Y$ connecting $p_1=(-1,0)$ to $p_2=(1,0)$, $p_3=(-1,-2)$ to $p_4=(1,-2)$ and $p_2$ to $p_1$, res\-pec\-tively; $\Gamma_2$ and $\Gamma_4$ are arc of trajectories of $X$ connecting $p_1$ to $p_3$ and $p_4$ to $p_2$, respectively
	(see Figure \ref{fig1-7}). The following result describes the behavior on $\mathit{\Lambda}$.
\end{example}

\begin{proposition}\normalfont\label{proposition-3-zonas}
	The non-empty set $\mathit{\Lambda}$ provided in Example 1 satisfies the following properties:
	\begin{enumerate}
		\item[(1)] it is chaotic for system \eqref{system3zones};
		\item[(2)] it is a minimal set.
	\end{enumerate}
	Moreover, there exist a trajectory $\widetilde{\Gamma}$ of (\ref{system3zones}) for which $\mathit{\Lambda}$ is its $\omega-$limit set.
\end{proposition}

The proof os Proposition \ref{proposition-3-zonas} is presented at the end of Section \ref{proof}.

Next example introduces a class of DVF satisfying hypothesis $(Z_3)$. In particular, we remark that this hypothesis is important to avoid the existence of non-trivial minimal sets as obtained in \cite{Buzzi-Carvalho-Euzebio18} and \cite{Buzzi-Carvalho-Euzebio18}. A generalization of the example is presented at Section \ref{linear}.

\begin{example1}
	The class of relay systems is considered in control theory, friction phenomena and other areas (see for instance \cite{AndronovEtAl66}, \cite{diBernardo08}, \cite{Flotz53} and references therein). According to \cite{diBernardo08} a planar relay system is a DVF that can be written in the general form 
	\begin{equation}\label{relay_system}
	X_{sgn(y)}:\left\{
	\begin{array}{lll}
	\dot{x} & = & Ax+Bu,\\
	y & = & C^Tx,\\
	u & = & sgn(y),
	\end{array}
	\right.
	\end{equation}
	where
	$$
	A=\left(
	\begin{array}{cc}
	a_{11}&a_{12}\\
	a_{21}&a_{22}
	\end{array}
	\right),\,
	B=\left(
	\begin{array}{c}
	b_{1}\\
	b_{2}
	\end{array}
	\right)\mbox{and}
	$$
	$$
	C=\left(
	\begin{array}{c}
	c_{1}\\
	c_{2}
	\end{array}
	\right)
	$$
	
	are real matrices. The sign function induces a separation region which in this case is the straight line $y=0$. Indeed, it is easy to see that $X_1$ and $X_2$ have at most one tangent point located at $\left(\frac{b_2}{c_2} a_{21},0\right)$ and $\left(\frac{-b_2}{c_2}a_{21},0\right)$, respectively. Therefore the class of relay systems has at most two tangent points with $y=0$.
\end{example1}

\begin{remark}\label{remark-sigma}
	Finally we remark that the shape of the switching manifold $\Sigma$ does not play any role throughout this paper. Still, it will useful to introduce a fixed parametrization of the curve $\s\cap K$, so consider
	$$
	\sigma:[\alpha,\beta]\subset\R\to\s\cap K,
	$$ 
	withe $\alpha<0<\beta$. We assume that $\s\cap K$ has the natural ordering given by the identification $\s\cap K=[\alpha,\beta]$. Moreover, for our purposes we only consider tangent points on $\s\cap K$ of even order. Otherwise the trajectories crosses $\Sigma$ always in the same direction as the regular case being the last situation considered along the paper.
\end{remark}




\section{Global analysis and au\-xi\-li\-aries results}\label{global-analysis}

This section is devoted to analyze the possible limit sets arising in the context of Theorem \ref{TPB_Euzebio-Juca}. In order to highlight the different kind of objects we are going to deal with, we split the study into some particular cases starting with the pseudo-cycles.

From now on we assume that trajectories of the DVF composing a maximal trajectory $\Gamma_Z(t,p)$ are contained on the compact set $K$ for positive values of $t$ and satisfy the following hypotheses according to Remarks \ref{Remark} and \ref{remark-sigma}.
\begin{itemize}
	\item[$H_1$:] $\g_Z^+(t,p)$ leaves $\overline{\s^{s,e}}$ and returns to it infinitely many times;
	\item[$H_2$:] the tangent points of $Z$ with $\s\cap K$ have even order.
\end{itemize}
We notice that these are non-empty assumptions such we will see on Section \ref{linear}. We also notice that occurring $H_1$ one may define a map from $\Sigma$ to itself , see for instance \cite{Buzzi2014}, so the first return on $\Sigma$ takes place. This notion is often used throughout this paper.

\subsection{Pseudo-cycles}\label{pseudo-cycles}

In this subsection we fully describe how a pseudo-cycle (of crossing, tangent or sliding type)  emerges as $\omega$-limit set of a maximal trajectory as stated in Theorem \ref{TPB_Euzebio-Juca}. In particular, the approach we consider allow us to better comprehend the topological structure of pseudo cycles as well as their asymptotic behavior.

Since the crossing pseudo-cycles were addressed in the Fundamental Lemma, next we will restrict our attention to the tangent and sliding pseudo-cycles obtained as $\omega-$limit sets for a maximal trajectory of a DVF that satisfy the hypothesis of Theorem \ref{TPB_Euzebio-Juca} as well as $H_1$ and $H_2$ above.\\
We remark the considered identification of $\s\cap K$ with the interval $[\alpha,\beta]\subset\R$ containing $0$ in its interior, according to Remark \ref{remark-sigma}.

\subsubsection{The regular-tangent case}\label{regular-tangent}

\begin{lemma}\label{aff1}
	Assume that hypotheses of Theorem \ref{TPB_Euzebio-Juca}, $H_1$ and $H_2$ hold, $X$ has a tangent point $p^-$ on the interior of $\s\cap K$ and $Y$ is transversal to $\s$ on $K$. Then $\omega(\g_Z(t,p))$ is either a tangent or a sliding pseudo-cycle.
\end{lemma}

\begin{proof}
	Without loss of generality we take $p^-=0$. Since $H_2$ holds we avoid odd tangency points so the following possibilities can occur.
	
	\begin{itemize}
		\item[(i)] $\Sigma^c\cap K=[\alpha,0)$ and $\Sigma^e\cap K=(0,\beta]$;
		\item[(ii)] $\Sigma^e\cap K=[\alpha,0)$ and $\Sigma^c\cap K=(0,\beta]$;
		\item[(iii)] $\Sigma^c\cap K=[\alpha,0)$ and $\Sigma^s\cap  K=(0,\beta]$;
		\item[(iv)] $\Sigma^s\cap K=[\alpha,0)$ and $\Sigma^c\cap K=(0,\beta]$.
	\end{itemize}
	
	\medskip
	
	Consider $p\in K$. First we assume that statement (i) holds. Since $H_1$ holds, we get $\Gamma_Z^+(t,p)\cap\Sigma\subset\Sigma^e\cup\{0\}$ being $0$ a visible tangency point for $X$ and a regular one for $Y$. Moreover, the trajectory $\Gamma_Z^+(t,p)$ only returns to $\Sigma$ by the tangency point $0$, otherwise such a trajectory would not return to $\s$. So let $q_1^-\geq 0$ be the point on $\Sigma$ such that the negative trajectory of $X$ starting at $0$ meets $\Sigma$ (see Figure \ref{fig1-15}). Since we are assuming infinitely many returns to $\Sigma$ it follows that $\Gamma^+_Z(t,p)$ only leaves $\Sigma^e$ by $q_1^-$. Consequently we obtain a tangent pseudo-cycle of $Z$ if $q_1^-=0$ (whose intersection with $\Sigma$ is the tangency point $0$) or a sliding pseudo-cycle of $Z$ if $q_1^->0$ and we are done.
	
	It is easy to see that statement (ii) leads to the analogous situation. Similarly, the proof of statements (iii) and (iv) is treated in the same way by reversing time in the trajectory $\Gamma_Z(t,p)$.
\end{proof}

\subsubsection{The tangent-tangent case assuming $p^-\neq p^+$}\label{p+difp-}

\medskip

Now each vector field $X$ and $Y$ contribute with tangency points $p^-$ and $p^+$ on the interior of $\s\cap K$, respectively.
Without loss of generality, assume that $p^-<p^+$. Now we have the following configurations of $\Sigma$ (we will to consider them in the proofs of Lemmas \ref{aff2} and \ref{aff3}):

\begin{itemize}
	\item[(i)] $\Sigma^c\cap K=[\alpha,p^-)\cup(p^+,\beta]$ and $\Sigma^s\cap K=(p^-,p^+)$;
	\item[(ii)] $\Sigma^c\cap K=[\alpha,p^-)\cup(p^+,\beta]$ and $\Sigma^e\cap K=(p^-,p^+)$;
	\item[(iii)] $\Sigma^s\cap K=[\alpha,p^-),\ \Sigma^c\cap K=(p^-,p^+)$ and $\Sigma^e\cap K=(p^+,\beta]$;
	\item[(iv)] $\Sigma^e\cap K=[\alpha,p^-),\ \Sigma^c\cap K=(p^-,p^+)$ and $\Sigma^s\cap K=(p^+,\beta]$.
\end{itemize}

We also split the analysis according to the visibility or not of $p^-$ and $p^+$. Since $H_1$ holds the invisible-invisible case cannot occur. In fact, if $p^-$ and $p^+$ are invisible tangency points then either $\Gamma_Z(t,p)\in\overline{\Sigma^s}$ or $\Gamma_Z(t,p)\notin\Sigma$ for $t$ sufficiently large since there is at most one tangent point of $X$ or $Y$ with $\s\cap K$. Then we only analyze the cases where at least one tangency point $p^-$ or $p^+$ is visible.


\paragraph{The invisible-visible sub-case}

\begin{lemma}\label{aff2}
	Assume that hypotheses of Theorem \ref{TPB_Euzebio-Juca}, $H_1$ and $H_2$ hold and both $X$ and $Y$ have non coincident tangency points $p^-$ and $p^+$, respectively, on the interior of $K\cap\s$ having opposite visibility. In this case $\omega(\g_Z(t,p))$ is a crossing, a tangent or a sliding pseudo-cycle.
\end{lemma}

\begin{proof}
	Without loss of generality we suppose that $p^-$ is an invisible tangency point and $p^+$ is a visible one. Assume that statement (i) holds. Since $H_1$ holds there exist a value $\widetilde{t}>0$ satisfying $\Gamma_Z(\widetilde{t},p)=p^+$ and a point $p_1^+\leq p^+$ which is the first return on $\Sigma$ in such way that the positive trajectory of $Y$ from $p^+$ meets $\Sigma$. (See Figure \ref{fig1-2}). We have three situations:
	
	\begin{itemize}
		\item If $p_1^+\in (p^-,p^+]$ then the regular-tangent case in the previous subsection applies. So we obtain a tangent or sliding pseudo-cycle of $Z$ if $p_1^+=p^+$ or $p^-<p_1^+<p^+$, respectively. (See Figure \ref{fig1-2} (a)).
		\item If $p_1^+=p^-$ then the Filippov vector field connects $p^-$ to $p^+$ since $H_1$ holds. Therefore pseudo-equilibria between $p^-$ and $p^+$ cannot occur in this situation. In this case we obtain a sliding pseudo-cycle. (See Figure \ref{fig1-2} (b)).
		\item If $p_1^+<p^-$ and setting $\Gamma_Z(t_0,p)=p_1^+$ from $H_1$ it follows that $\Gamma_Z(t,p)$ returns to $\Sigma$ at $p_2^+\in(p^-,p^+]$ for some $t>t_0$ following $X$. If $p^-<p_2^+<p^+$ then the last bullet applies and we obtain a sliding pseudo-cycle. If $p_2^+=p^+$ then $\Gamma_Z(t,p)$ is a tangent pseudo-cycle touching $\Sigma$ at two points, one of them a sewing one and other the tangency point $p^+$.
	\end{itemize}
	
	\begin{center} 
		\includegraphics[scale=0.5]{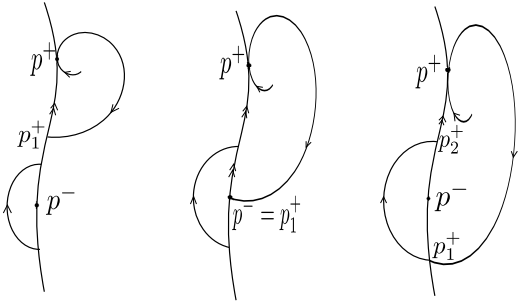}
		\captionof{figure}{The invisible-visible case of statement (i).}\label{fig1-2}
	\end{center}
	
	Statement (ii) is treated analogously to the previous case by reversing time on the trajectory $\Gamma_Z(t,p)$. Lastly, assume that statements (iii) or (iv) hold. In those cases, if $\Gamma_Z(t_0,p)\in\Sigma^s$ then $\Gamma_Z(t,p)\notin\Sigma^e$ for $t>t_0$ since we are considering the invisible-visible sub-case. In fact, if $p^-$ is an invisible tangency point and $p^+$ is a visible one it follows that  $\Gamma_Z(t,p)\cap\Sigma\subset [p^+,\beta]\subset\s^e$ for $t$ sufficiently large. Then the tangent-transversal case applies and therefore we obtain the same possibilities for limit set obtained in that situation.
	Interchanging the visibility of $p^-$ and $p^+$ we obtain analogous conclusions.
\end{proof}



\paragraph{The visible-visible sub-case}

\begin{lemma}\label{aff3}
	Assume that hypotheses of Theorem \ref{TPB_Euzebio-Juca}, $H_1$ and $H_2$ hold, $X$ and $Y$ have non coincident tangency points $p^-$ and $p^+$ on the interior of $K\cap\s$, respectively, being both visible. Then $\omega(\g_Z(t,p))$ is a crossing, a tangent or a sliding pseudo-cycle.
\end{lemma}

\begin{proof}
	Let $p$ be a point on $K$. Assume that statement (i) hold. From hypotheses $H_1$ and since both $p^+$ and $p^-$ attract trajectories of the Filippov vector field there exist at least one pseudo equilibrium point between $p^-$ and $p^+$. Moreover the trajectory of a point $q\in\Sigma^s$ only leaves $\Sigma$ by $X$ at $p^-$ or by $Y$ at $p^+$. Suppose that $\Gamma_Z(t,p)$ leaves $\overline{\Sigma^s}$ at $p^-=\Gamma_Z(\widetilde{t},p)$ for some $\widetilde{t}>0$ and let $p_1^-\geq p^-$ be on $\Sigma$ be the first point where the positive trajectory of $X$ from $p^-$ meets $\Sigma$. 
	
	Now we have three situations to consider in terms of the tangency points. We start assuming $p_1^-\geq p^+$ (see Figure \ref{fig1-8}). Let $p_2^-\in\Sigma$ be the first point such that the positive trajectory of $Y$ from $p_1^-$ meets $\Sigma$. Then $p^-< p_2^-\leq p^+$ since $\Gamma_Z(t,p)$ by $H_1$. In fact if $p_2^-\leq p^-$ then $\Gamma_Z(t,p)$ only touches $\Sigma$ in sewing points for sufficiently large values of $t$ and therefore the Fundamental Lemma applies. Moreover, clearly the situation $p_2^-> p^+$ cannot occur given the statement (i). If $p_2^-=p^-$ then we obtain a tangent pseudo-cycle whose intersection with $\Sigma$ is the tangency point $p^-$ and the tangency point $p_1^-=p^+$ or a sewing point $p_1^->p^+$.  We get:
	\begin{itemize}
		\item If $p_2^-=p^+$ then we obtain a tangent pseudo-cycle whose intersection with $\Sigma$ is the tangency point $p^+$.
		\item If $p^-<p_2^-<p^+$ then $p_2^-$ cannot be located between two pseudo-equilibrium points once $H_1$ holds. Consequently the Filippov vector field connects $p_2^-$ to either $p^-$ or $p^+$. The first case is exhibited in the Figure \ref{fig1-8} (a). In the second case the positive trajectory of $Y$ connects $p^+$ to $\Sigma^s$ at $p_1^+\geq p_2^-$ (see Figure \ref{fig1-8} (b)). In the first case we obtain a sliding pseudo-cycle. In the second case we obtain a sliding pseudo-cycle topologically different from the first one, or a tangent pseudo-cycle whose intersection with $\Sigma$ is the tangency point $p^+$.
	\end{itemize}

	\begin{center} 
		\includegraphics[scale=0.5]{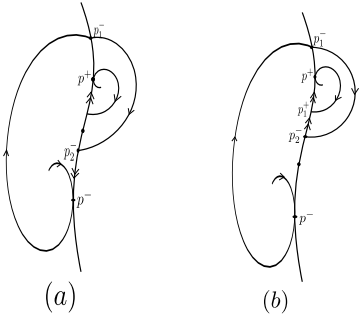}
		\captionof{figure}{The case $p_1^-\geq p^+$.}\label{fig1-8}
	\end{center}
	
	Now we assume $p^-<p_1^-<p^+$. As before, in this case the Filippov vector field must connect $p_1^-$ to either $p^-$ or $p^+$. In the first situation the analysis is the same of the regular-tangent case. In the second one there exist $p_1^+\leq p^+$ the first point on $\Sigma$ such that the positive trajectory of $Y$ from $p^+$ meets $\Sigma$. If $p_1^+<p^-$ then the positive trajectory of $X$ connects $p_1^+$ to $p_2^+\in (p_1^-,p^+]$ since $H_1$ holds. In this case we obtain a sliding pseudo-cycle if $p^-<p_2^+<p^+$ or a tangent pseudo-cycle if $p_2^+=p^+$ (see Figure \ref{fig1-3} (a)). On the other hand, if $p_1^+=p^-$ then we obtain a sliding pseudo-cycle (see Figure \ref{fig1-3} (b)). If $p^-<p_1^+<p^+$ then the Filippov vector field connects $p_1^+$ to either $p^-$ or $p^+$ (see Figure \ref{fig1-3} (c) for the first case) and we obtain a sliding pseudo-cycle. Finally, if $p_1^+=p^+$ then we obtain a tangent pseudo-cycle whose intersection with $\Sigma$ is the tangency point $p^+$ so we are done.
	
	If $p_1^-=p^-$ we obtain a tangent pseudo-cycle whose intersection with $\Sigma$ is the tangency point $p^-$.
	
	\begin{center} 
		\includegraphics[scale=0.5]{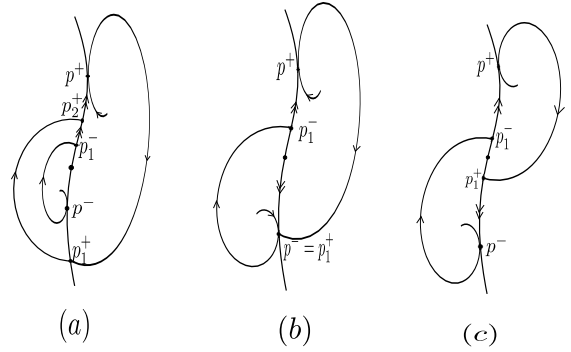}
		\captionof{figure}{The case $p^-<p_1^-<p^+$ where the Filippov vector field connects $p_1^-$ to $p^+$. $(a)\ p_1^+<p^-,\ (b)\ p_1^+=p^-$ and $(c)\ p_1^+>p^-$ and Filippov vector field connects $p_1^+$ to $p^-$.}\label{fig1-3}
	\end{center}
	
	The statement (ii) is treated analogously to the previous case by reversing time in the trajectory $\Gamma_Z(t,p)$. Finally, if statements (iii) or (iv) hold, then we have $\Gamma_Z(t,p)\notin\Sigma^e$ for $t>t_0$ when $\Gamma_Z(t_0,p)\in\Sigma^s$. In this situation the tangent-transversal case applies.
\end{proof}



\subsubsection{The tangent-tangent case assuming $p^-=p^+$}

\begin{lemma}\label{aff4}
	Assume that hypotheses of Theorem \ref{TPB_Euzebio-Juca}, $H_1$ and $H_2$ hold, $X$ and $Y$ have coincident tangency points on $K\cap\s$, $p^- =p^+$. Then $\g_Z(t,p)$ can be taken such that $\omega(\g_Z(t,p))$ is a crossing pseudo-cycle.
\end{lemma}

\begin{proof}
	Without loss of generality suppose that $p^-=p^+=0$. Now we get the following possible configurations:
	
	\begin{itemize}
		\item[(i)] $\Sigma^c\cap K=[\alpha,\beta]\setminus\{0\}$;
		\item[(ii)] $\Sigma^s\cap K=[\alpha,0)$ and $\Sigma^e\cap K=(0,\beta]$;
		\item[(iii)] $\Sigma^e\cap K=[\alpha,0)$ and $\Sigma^s\cap K=(0,\beta]$.
	\end{itemize}
	
	If statement (i) holds the crossing pseudo-cycles are obtained as in the Fundamental Lemma. So assume that statement (ii) holds. We claim that in this case there is no pseudo-cycle. Indeed let $\g$ be a closed maximal trajectory of $Z$ such that $\g\cap\s\neq\emptyset$. Assume that it does not contain neither equilibria nor pseudo-equilibria. We will show that $\g$ is neither positively nor negatively invariant and therefore it does not satisfies condition $(i)$ of Definition \ref{def_pseudo-cycle}.

	Notice that the trajectories of the Filippov vector field $Z^\s$ on $[-\varepsilon,0)$ and $(0,\varepsilon]$ have the same direction for $\varepsilon>0$ sufficiently small. If $\g\cap\s^s\neq\emptyset$ then $\g$ is not negatively invariant because $\g$ is a meager set while the saturation of a sliding segment in past time has non empty interior. If $\g\cap\s^e\neq\emptyset$ then the previous argument applies and $\g$ is not positively invariant. If $(\g\cap\s^s)\cup(\g\cap\s^e)=\emptyset$ then $0$ is a visible tangent point for at least one vector field $X$ or $Y$. Since the Filippov vector field is extendable beyond the boundary of $\s^{e,s}$ we have that $\g$ is neither negatively nor positively invariant. It follows that $\g$ is not a pseudo-cycle. The case in that statement (iii) holds is entirely analogous.
\end{proof}



\subsection{Mild Pseudo-cycles and Chaotic Sets of Type I and II}\label{mild-pc_and_chaotics}

In this section we study the mild pseudo-cycles of type I, II and III and its relation with chaotic sets for a DVF as in Theorem \ref{TPB_Euzebio-Juca}.

\begin{proposition}\normalfont
	Every mild pseudo-cycle of type II or III is chaotic of type I or II.
\end{proposition}

\begin{proof}
	Let $\g$ be a mild pseudo-cycle of type I or II, that is, $(i)$ $\g$ is a closed maximal trajectory of $Z$, $(ii)$ $\g\cap\s\neq\emptyset$, $(iii)$ $\g$ does not contain neither equilibria nor pseudo-equilibria and $(iv)$ $\g$ contains a proper maximal trajectory.
	
	Since (i) holds we have that $int(\g)=\emptyset$. Now we must prove that (a) $Z$ is topologically transitive on $\g$, (b) $Z$ exhibits sensitive dependence on $\g$ and (c) the periodic orbits of $Z$ are dense.
	
	Being $\g$ itself a closed maximal trajectory then (c) holds. To prove (a) and (b) let $\widetilde{\g}\subset\g$ be a proper maximal trajectory and consider $p\in\widetilde{\g}$ a point such that there is no uniqueness of trajectory in $\g$, so (a) holds. Moreover, let $\gamma_1(t,p)=\g$ and $\gamma_2(t,p)=\widetilde{\g}$ be these maximal trajectories, with $\gamma_1(0,p)=\gamma_2(0,p)=p$. Then, we have that $2r=\displaystyle\sup_{t>0}d(\gamma_1(t,p),\gamma_2(t,p))>0$. Given $x\in\g$ if $y\in\g$ is sufficiently close to $x$ we get $\g_Z(t,p)=p=\g_Z(\overline{t},p)$ for $\overline{t}$ sufficiently close to $t$. Therefore we obtain the sensitive dependence in $\g$ from $d(\gamma_1(t,p),\gamma_2(t,p))>r$ for some $t>0$.
\end{proof}

\begin{lemma}\label{aff5}
	Assume that hypotheses of Theorem \ref{TPB_Euzebio-Juca}, $H_1$ and $H_2$ hold and $X$ and $Y$ have coincident tangency points $p^-$ and $p^+$ on $K\cap\s$. Then $\g_Z(t,p)$ can be taken such that $\omega(\g_Z(t,p))$ is a mild pseudo-cycle of type I, II or III.
\end{lemma}

\begin{proof}
	Without loss of generality suppose that $p^- = p^+ = 0$. Since this is the unique tangency point of $Z$ with $\s$ on $K$ being of even order for both vector fields $X$ and $Y$, we have the following possible configurations:
	
	\begin{itemize}
		\item[(i)] $\Sigma^c\cap K=[\alpha,\beta]\setminus\{0\}$;
		\item[(ii)] $\Sigma^s\cap K=[\alpha,0)$ and $\Sigma^e\cap K=(0,\beta]$;
		\item[(iii)] $\Sigma^e\cap K=[\alpha,0)$ and $\Sigma^s\cap K=(0,\beta]$.
	\end{itemize}
	
	Since $H_1$ holds, statement (i) cannot occur. Suppose that statement (ii) holds. Again from $H_1$ the double tangency $0$ cannot be elliptic. In fact, if $0$ is an elliptic tangency point then it attracts the Filippov trajectory on $\s^e$ and repels its on $\s^s$. Since $0$ is the unique tangent point on $\s\cap K$ and the positive trajectory of $\g=\g_Z(t,p)$ is contained entirely on $K$, $H_1$ is not satisfied, that is, it cannot occur infinitely many transitions of $\g$ with $\Sigma$ by $\overline{\s^e\cup\s^s}$. Therefore $0$ is a parabolic or hyperbolic tangent point.
	
	Notice that the Filippov vector field on $[-\varepsilon,0)$ and $(0,\varepsilon]$ have the same direction for $\varepsilon>0$ sufficiently small. We split the analysis into the following cases:
	
	\begin{itemize}
		\item[(1)] $\Gamma_Z^+(t,p)$ visits $\Sigma^s\cup\s^e$ a finite number of times;
		\item[(2)] $\Gamma_Z^+(t,p)$ visits $\Sigma^s$ a finite number of times and visits $\s^e$ an infinite number of times;
		\item[(3)] $\Gamma_Z^+(t,p)$ visits $\Sigma^s$ an infinite number of times and visits $\s^e$ a finite number of times;
		\item[(4)] $\Gamma_Z^+(t,p)$ visits both $\Sigma^s$ and $\s^e$ an infinite number of times.
	\end{itemize}
		
	Suppose that situation (1) occurs.  Since $H_1$ holds there exist $t_0>0$ such that $\g_Z(t,p)\cap\s=\{0\}$ for $t>t_0$ and $\g_Z^+(t+t_0,p)$ is a periodic orbit of $X$ or $Y$ or a union between them. In the first two cases we obtain a mild pseudo-cycle of type I and in the third case we obtain a mild pseudo-cycle of type III (see Figure \ref{fig1-5} (c)).
	
	\begin{center} 
		\includegraphics[scale=0.35]{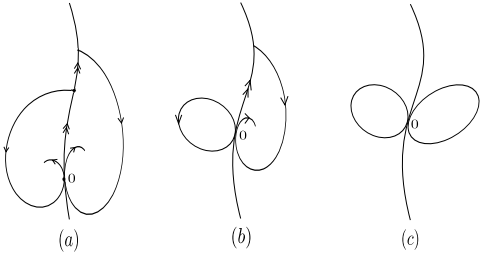}
		\captionof{figure}{Examples of mild pseudo-cycles.\label{fig1-5}}
	\end{center}
	
	Now, suppose that situation (2) occurs. Then there exist $t_0>0$ such that $\g_Z(t,p)\notin\s^s$ for $t>t_0$. Since there are infinitely many transitions of $\g_Z(t,p)$ from and to $\s$, the returns of $\g_Z(t,p)$ to $\s$ with $ t>t_0$ occur through the visible tangency point $0$. If $\g_Z(t,p)$ returns to $0$ by $X$ or $Y$ there is $q_1^\pm\in\overline{\s^e}$ the point for which the past of $0$ by $X$ or $Y$ meets $\s$. Moreover, at least one of the points $q_1^-$ or $q_1^+$ belong to $\s^e$ and, for $t>t_0$ we have that $\g_Z(t,p)$ leaves $\s^e$ only by $q_1^-$ according to $X$ or by $q_1^+$ according to $Y$, see Figures \ref{fig1-5} (a) and (b). The tangent point $0$ attracts the trajectories of the Filippov vector field on $\s^s$ and repels them on $\s^e$ in a neighborhood of $0$ in this case. Thus we obtain a mild pseudo-cycle of type I if there exist only one of $q_1^-$ or $q_1^+$ and a mild pseudo-cycle of type III if there exists both $q_1^-$ and $q_1^+$. The case (3) is analogous by reversing orientation on time.

	Finally, suppose that situation (4) occurs. Let $p_s'$ and $p_e'$ be the frontier of the compact set $K\cap\s$ in $\s^s$ in $\s^e$, respectively. Fix $p_s=\max\{\widetilde{p_s},p_s'\}$ and $p_e=\min\{\widetilde{p_e},p_e'\}$, where $\widetilde{p_s}\in\s^s$ and $\widetilde{p_e}\in\s^e$ are the nearest pseudo-equilibrium points of $0$.  If there is no pseudo-equilibrium on $K\cap\s$ then $p_s=p_s'$ and $p_e=p_e'$, respectively. Therefore $\g_Z(t,p)\cap\s^s\subset [p_s,0)$, $\g_Z(t,p)\cap\s^e\subset (0,p_e]$ and the extended Filippov vector field connects $p_s$ to $p_e$ in future time. 
	Thus we can take $\g_Z^+(t+t_0,p)$ as being a closed trajectory, for some $t_0>0$, that is neither positively nor negatively invariant. Follows that $\omega(\g_Z(t,p))$ is a mild pseudo-cycle of type II or III and we are done.
\end{proof}


\subsection{Chaotic set of type III and pseudo-graphs}\label{chaotic}

In this section we study the chaotic sets of type III and those pseudo-graphs that are $\omega$-limit of some maximal trajectory in the sense of Theorem \ref{TPB_Euzebio-Juca}.

Suppose that $(\s^s\cup\s^e)\cap K\neq\emptyset$. Let $\Omega_g$ be the set of pseudo-graphs and $\Omega^*\subset\Omega_g$ be the subset of pseudo-graphs contained on $K$ that are $\omega$-limit of some maximal trajectory of $Z$ whose positive part is contained on $K$. That distinction is necessary because some pseudo-graphs are not $w-$limit of any trajectory.

Next result provides a partial description of the elements of $\Omega^*$.

\begin{proposition}\label{prop-pseudo-grafico}
	Assume that hypotheses of Theorem \ref{TPB_Euzebio-Juca} hold and let $\g\in\Omega^*$ such that $\g\cap (\s^s\cup\s^e)\neq\emptyset$, that is, $\g$ is a pseudo-graph of $Z$ on $K$ which is the $\omega$-limit of a maximal trajectory of $Z$ with sliding motion. Then, $X$ and $Y$ have coincident tangency points $p^-=p^+$ on $K$.
\end{proposition}

\begin{proof}
	Initially suppose there exist $q\in\g\cap\s^s$. Since $\g=\omega(\g_Z(t,p))$ there exist a sequence $p_n=\g_Z(t_n,p)$ converging to $q$, where $t_n\to\tau^+(p)$ when $n\to+\infty$. Notice that there exist a neighborhood $V_q$ of $q$ such that every point in $V_q$ meets $\s^s$ in future finite time. So take $q_n=\g_Z(t_n^*,p_n)\in\s^s$ with $t_n^*>0$. Since $p_n\to q$ we get $q_n\to q$. Moreover, the future of all points $q_n$ for $n$ sufficiently large visits a visible tangency point, that we suppose without loss of generality to be $p^-$. So it follows that there are no pseudo-equilibria between $q$ and $p^-$. If $p^-\neq p^+$ then $\s^s$ does not connects with $\s^e$ and the trajectory $\g_Z(t,p)$ is regular periodic whose $\omega$-limit is itself and therefore $\omega(\g_Z(t,p))$ is not a pseudo-graph. An analogue contradiction is obtained for a pseudo-graph $\g$ intersecting $\s^e$ by reversing time or in the regular-tangent case. Therefore $X$ and $Y$ have coincident tangency points $p^-=p^+$ on $K$.
\end{proof}

Next result is a partial converse of \ref{prop-pseudo-grafico}.

\begin{lemma}\label{aff6}
	Assume that hypotheses of Theorem \ref{TPB_Euzebio-Juca}, $H_1$ and $H_2$ hold and $X$ and $Y$ have coincident tangency points $p^-$ and $p^+$ on $K\cap\s$. Then $\g_Z(t,p)$ can be taken such that $\omega(\g_Z(t,p))$ is a chaotic set $\Lambda$ of type III. Moreover, if $\mathcal{G}=\partial\Lambda\cup[\Lambda\cap\Sigma]$ contains any equilibrium or pseudo-equilibrium, then $\mathcal{G}$ is a pseudo-graph and $\g_Z(t,p)$ can be taken such that $\omega(\g_Z(t,p))$ $=\mathcal{G}$.
\end{lemma}

\begin{proof}
	We will split the proof into four steps. Initially we will construct the set $\Lambda$ (Step 1) and then we show that there exists maximal trajectories in such that $\Lambda$ and $\mathcal{G}$ are their $\omega$-limit sets (Steps 2 and 3, respectively). Finally we show that $\Lambda$ is a chaotic set of type III (Step 4).

\smallskip
\textit{\bf Step 1: Construction of $\Lambda$.}

	Without loss of generality suppose that $p^-=p^+=0$. Notice that from hypotheses $H_1$, $0$ cannot be an elliptic tangency point as we verify in the beginning of the proof of Lemma \ref{aff4}. Then we will construct the set $\Lambda$ assuming that $0$ is parabolic. The hyperbolic case will follow naturally from this case. Let $p_s$ and $p_e$ be points such as in the end of the proof of Lemma \ref{aff4} and let $q_e^+\in(0,p_e]$ be determined as follows (see Figure \ref{fig1-6}). If $q\in(0,q_e^+)$ then the future of $q$ by the flow of $Y$ meets $[p_s,0)\in\s^s$, say at $q^+$, and notice it cannot occur to the another initial condition $q>q_e^+$ on $\s^e$ (except possibly if $q_e^+=p_e$).  Notice also that the region delimited by the arc of trajectory of $Y$ with boundary $q\in(0,q_e^+)$ and $q^+$ joining the segment $[q^+,q]\subset\s$ contains no equilibrium point of $Y$. The portion of $\Lambda$ contained on $\s^+$ is then defined as being the closure of the union of these regions.

\begin{center} 
	\includegraphics[scale=0.5]{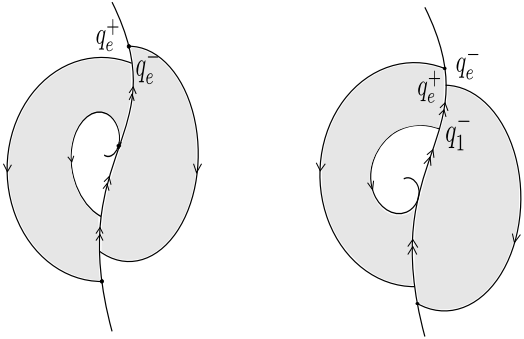}
	\captionof{figure}{Chaotic set of type III.\label{fig1-6}}
\end{center}

	Now we proceed to describe the portion of $\Lambda$ contained on $\s^-$ containing the visible tangency point. If neither the past or future of the tangency point $0$ by $X$ meets $[p_s,p_e]$, then $\Lambda\cap\s^-=\emptyset$. Suppose that the past of $0$ by $X$ meets $(0,p_e)$ at $q_1^-$ and let $q_e^-\in[q_1^-,p_e]$ be defined as follows. If $q\in(q_1^-,q_e^-)$ then the future of $q$ by $X$ meets $[p_s,0)$, say at $q^-$. Again, it cannot happen to the another initial condition $q>q_e^-$ (except possibly if $q_e^-=p_e$). Notice that in this case there is no equilibrium point in the region on $\s^-$ delimited by the arc of trajectory of $X$ with boundaries $q$ and $q^-$ joining the segments $[q^-,0]\cup[q_1^-,q_e^-]$ (see Figure \ref{fig1-6}). The portion of $\Lambda$ contained on $\s^-$ is then defined as being the closure of the union of these regions.
	
	In case the future of $0$ by the trajectory of $X$ meets $(p_s,0)$, say at $p_s^-$, the approach is similarly to the previous one by looking to the past of trajectories with initial condition on $[p_s,q_s^-)$.

\smallskip
\textit{\bf Step 2: $\Lambda$ is the $\omega$-limit of some maximal trajectory achieving the hypothesis of Lemma \ref{aff5}.}

	From the previous construction we see that $0\in\Lambda$ is the common boundary between the sliding and the escaping region. Moreover, for any interior points $r,s$ of $\Lambda$ there exists $t_r>0$, $t_s<0$ and trajectories $\Gamma_Z^r$ and $\Gamma_Z^s$ such that $\Gamma_Z^r(t_r,r)=\Gamma_Z^s(t_s,s)=0$. Hence any two points on the interior can be connected by an arc of some maximal trajectory. Therefore the closure of $\Lambda$ is the $\omega-$limit set of a maximal trajectory.

\smallskip
\textit{\bf Step 3: $\partial\Lambda$ is the $\omega$-limit of some maximal trajectory achieving the hypothesis of Lemma \ref{aff5}.}

	Consider an arbitrary point $p\in\Lambda$ and sequences $\{p_n\}\subset(0,q_e^+)$ and $\{q_n\}\subset(0,q_e^-)$ converging to $q_e^+$ and $q_e^-$, respectively. Consider also the maximal trajectory $\g_Z(t,p)$ given as follows. Starting at $p$, there exist $t_0>0$ such that $\g_Z(t_0,p)=0$. Due to the recurrence through $0$ there exists two sequences $0<t_0<s_1<t_1<s_2<t_2<\cdots<s_n<t_n<\cdots$ with $t_n\to\infty$ such that $\g_Z(s_n,p)=p_n,\ \g_Z(t_n,p)=q_n$. Therefore $\g_Z(t,p)$ leaves $\s^e$ by $Y$ at $p_n$ for $t=s_n$. By continuity of trajectories respect to initial conditions and using the compactness of $\Lambda$ it follows that $\g_Z(t,p)$ can be taken such that $\omega(\g_Z(t,p))=\mathcal{G}\subset\partial\Lambda\subset\Lambda$. Moreover, if $\mathcal{G}$ contains no equilibrium or pseudo equilibrium we have that it is a pseudo-graph.

\smallskip
\textit{\bf Step 4: $\Lambda$ is a chaotic set of type III.}

	Let $U$ and $W$ be non-empty open subsets of $\Lambda$. For a point $p\in U$ there exist a positive time $t_0>0$ such that $\Gamma_Z(t_0,p)=0$. Since the positive trajectories of $Z$ contained on $\Lambda$ starting at $0$ are dense on $\Lambda$ by construction, there exist $t_1> 0$ such that $\Gamma_Z(t_1,0)\in W$. Then $\Gamma_Z(t_0+t_1,p)\in W$. Therefore $\Lambda$ is topologically transitive. Analogously we show the density of periodic orbits on $\Lambda$. For the sensitive dependence of $Z$ on $int(\Lambda)$, we notice that for $x$ sufficiently close to $y$ on $int(\Lambda)$, there exist a time $t>0$ sufficiently close to $s>0$ such that $\Gamma_Z(t,x)=\Gamma_Z(s,y)=p_0$. If $r=\displaystyle\frac{1}{2}diam(\Lambda)$ then there exists trajectories $\Gamma_Z(t,x)$ and $\Gamma_Z(t,y)$ of $Z$ such that $d(\Gamma_Z(T,x),\Gamma_Z(T,y))>r$ for some $T>0$. Moreover, clearly the interior of $\Lambda$ is non-empty.
\end{proof}

\begin{remark}\normalfont
	We see that $\Lambda$ cannot be a minimal set since it is either non-invariant or there always exist a non-empty invariant proper subset. 
\end{remark}

As an immediate consequence of the previous discussion we have the following result:

\begin{proposition}\label{cor_chaos}\normalfont 
	Assume that hypotheses of Theorem \ref{TPB_Euzebio-Juca} hold and that $Z$ is chaotic on the compact set $K$. Then, 
	\begin{itemize}
		\item[(i)] there exist a parabolic or hyperbolic (double) tangency point on $K$;
		\item[(ii)] $\Sigma^c\cap K=\emptyset$;
		\item[(iii)] there exist a maximal trajectory $\Gamma_Z(t,p)$ and a strictly increase sequence of times $t_n,s_n>0$ such that $\g_Z(t_n,p)\in\s^e$ and $\g_Z(s_n,p)\in\s^s$.
	\end{itemize}
\end{proposition}

In particular, Proposition \ref{cor_chaos} states that, under the hypotheses of Theorem \ref{TPB_Euzebio-Juca}, if a DVF is chaotic them it is structural unstable. It holds due to the connection of tangency points on $\Sigma$ which is clearly destroyed by suitable small perturbations.


%
%
%
%
%

\section{Proof of the main results}\label{proof}

Now we synthesize the analysis on the previous section in order to establish the proof of the main results of the paper.

\begin{proof}[Proof of Fundamental Lemma.]
	From Remark \ref{Remark} it is sufficient we consider the planar DVF case. First assume that $\Gamma_Z(t,p)\cap\overline{\Sigma^e\cup\Sigma^s}\neq\emptyset$. In this case, by using the proof provided in \cite{Buzzi-Carvalho-Euzebio18} we get that $\omega(\Gamma_Z(t,p))$ is one of items from (i) to (vii).
	
	Assume now that 
	$\Gamma_Z(t,p)\notin\overline{\Sigma^e\cup\Sigma^s},\ \forall t>t_0$. Since we are interested in limit sets we only need to consider $t\in\R$ arbitrarily large, therefore we also get items (i) to (vii) in the last situation.
	
	On the other hand, given a fixed $T\in\mathbb{R}$, if $\Gamma_Z(t,p)\in\Sigma^e\cup\Sigma^s,\ \forall t>T$ and since $\Gamma_Z^+(t,p)$ is contained on a compact set it follows that $\omega(\Gamma_Z(t,p))$ must accumulate in a point on $q\in\overline{\Sigma^s\cup\Sigma^e}$. If such a point belongs to $\Sigma^s\cup\Sigma^e$ then it is a pseudo-equilibrium, so we get statement $(viii)$. If $q$ belongs to the boundary of $\Sigma^s\cup\Sigma^e$ then $q$ must be a tangency point. From that point we could concatenate other trajectories of $X$ or $Y$ to $q$ unless $q$ is an invisible tangency point for a vector field and a visible tangency of odd order to another, being $q$ reached in finite positive time in the last situation. Under this configuration if $q$ is an attractor point to the sliding segment, then it is a tangency of type II. Therefore we have statement $(ix)$. If $q$ is a repeller so by the previous argument it should accumulate in a pseudo-equilibrium point and we are done. 
\end{proof}

\begin{proof}[Proof of Theorem \ref{TPB_Euzebio-Juca}]
	Let $K\subset M$ be a non-empty compact set contained on a coordinate neighborhood of $M$ and take $p\in K$ such that the positive maximal trajectory $\Gamma_Z^+(t,p)$ is contained on $K$. From Remark \ref{Remark} it is sufficient to consider $K\subset\R^2$. 	First assume that $\Gamma_Z(t,p)\notin \overline{\Sigma^e\cup\Sigma^s}$ for $t$ sufficiently large. If the orbit does not touch $\Sigma$ again then according to Fundamental Lemma $\omega(\Gamma_Z(t,p))$ is either an equilibrium, a periodic orbit or a graph of $Y$ or $X$ so we get items $(i)$, $(ii)$ and $(iii)$. Otherwise $\omega(\Gamma_Z(t,p))$ is either a crossing pseudo-cycle, a pseudo-graph, a mild pseudo cycle of type II or a tangency of type I of $Z$, so now we get (partially) items $(v)-(viii)$. On the other hand, if $\Gamma_Z(t,p)\in\overline{\Sigma^e\cup\Sigma^s}$ for $t$ sufficiently large then again according to Fundamental Lemma $\omega(\Gamma_Z(t,p))$ is either a pseudo-equilibrium of $Z$ or a tangency of type II so we get items $(iv)$ and $(viii)$.
	
	Suppose now that for any $T>0$ there exists $s,t>T$ such that $\Gamma_Z(s,p)\in\overline{\Sigma^e\cup\Sigma^s}$ and $\Gamma_Z(t,p)\notin\Sigma$. Notice that under this assumption (a) the existence of at least one tangency point is guaranteed on $\Sigma\cap K$ for $X$ or $Y$ or both and (b) when there exist two distinct tangency points they cannot be simultaneously invisible (see Subsection \ref{p+difp-}). Moreover, by hypothesis these vector fields have at most one tangency point on $\Sigma\cap K$. Then analyzing $X$ and $Y$ on $\Sigma\cap K$ we obtain four different situations:\\
	$\bullet$ 1) {\it $X$ has a tangency point and $Y$ is regular (or reciprocally)}. In this case by Lemma \ref{aff1} we have that $\omega(\Gamma_Z(t,p))$ is a tangent or a sliding pseudo-cycle contained entirely in $\Sigma^+$ or $\Sigma^-$ so we get item $(v)$ of the theorem. \\
	$\bullet$ 2) {\it $X$ and $Y$ have non coincident tangency points of opposite visibility}. In this case by Lemma \ref{aff2} the $\omega(\Gamma_Z(t,p))$ is either a crossing or tangent or sliding pseudo-cycle then we get item $(v)$.\\
	$\bullet$ 3) {\it $X$ and $Y$ have non coincident visible tangency points}. In this case by Lemma \ref{aff3} the $\omega(\Gamma_Z(t,p))$ is either a crossing or tangent or sliding pseudo-cycle therefore we get item $(v)$.\\
	$\bullet$ 4) {\it $X$ and $Y$ have coincident tangency points} In this case the $\omega(\Gamma_Z(t,p))$ is either a crossing pseudo-cycle (Lemma \ref{aff4}) or a mild pseudo-cycle of type I, II or III (Lemma \ref{aff5}) or a chaotic set of type III or it is a pseudo-graph of $Z$ (Lemma \ref{aff6}). In these cases we get items $(v)$, $(vi)$, $(ix)$ and $(vii)$, respectively.,
	The four previous cases contemplate all the objects listed in Theorem \ref{TPB_Euzebio-Juca} and so the proof is ended.
\end{proof}

As commented in Section \ref{heart} the pseudo-cycles and pseudo-graphs listed in Theorem \ref{TPB_Euzebio-Juca} may have different topological types. Indeed, the pseudo cycles obtained in case 1) of the proof are contained on $\Sigma^+$ or $\Sigma^-$ and their intersection with $\Sigma$ have only one component. On the other hand, the pseudo-cycles obtained in case 2) are  contained on either $\Sigma^+$ or $\Sigma^-$ or in both $\Sigma^\pm$ and their intersection with $\Sigma$ has again one component. The pseudo-cycles obtained in case 3) of the proof are contained on $\Sigma^+$ or $\Sigma^-$ or in both $\Sigma^\pm$. However, now the intersection of the pseudo-cycle with $\Sigma$ has one or two components. Concerning the pseudo-graphs, those obtained from Fundamental Lemma have neither sliding nor escaping segments while the pseudo-graphs from case 4) have both type of regions simultaneously. In any case, it is easy to see that those object have distinct topological types.

%
%


\begin{proof}[Proof of Theorem \ref{proprerties_of_lambda}]
	By subsection \ref{chaotic} there is a non-minimal chaotic set $\Lambda$ of type III for $Z$.  Given $p\in int(\Lambda)$ there exists $t_0,\ t_1>0$ such that $\Gamma_Z(t_0,p)=p_0$ and $\Gamma_Z(t_1,p_0)=p$. Then every point in $int(\Lambda)$ is periodic. In particular the periodic orbits of $Z$ are dense in $\Lambda$. Moreover, since there exist a dense trajectory in $\Lambda$ passing through $p_0$, it follows that there exist a dense trajectory in $\Lambda$ passing through $p\in int(\Lambda)$.  Since $\Gamma_Z(t_0,p)=p_0$ with $t_0>0$ it also follows that $\omega(\Gamma_Z(t,p))=\omega(\Gamma_Z(t,p_0))=\Lambda$ and therefore $p\in int(\Lambda)$ is a non-trivial recurrence point. Finally, for all $V$ neighborhood of $p$ and $t_0>0$ there exist $t>t_0$ such that $\Gamma_Z(t,p)\in V$, where $\Gamma_Z(t,p)$ is a dense trajectory on $\Lambda$. Therefore, $p$ is a non-trivial non wandering point.
\end{proof}

It remains to prove Proposition \ref{proposition-3-zonas} which we do next.

\begin{proof}[Proof of Proposition \ref{proposition-3-zonas}]
	First notice that $\mathit{\Lambda}$ is compact and invariant by construction. Moreover, notice that for every point $p\in\mathit{\Lambda}$ there exist a maximal trajectory $\Gamma_Z(t,p)$ passing through $(-1,0)$ in future time. Also  for each $q\in\mathit{\Lambda}$ there exist a maximal trajectory $\Gamma_Z(s,(-1,0))$ passing through $q$ in future time. Then for each $p,q\in\mathit{\Lambda}$ there exist a value $s+t>0$ such that $\Gamma_Z(s+t,p)=q$. That proves statement (2). In particular, $\Gamma$ is topologically transitive and the periodic orbits are dense on it. 
	To prove the sensitive dependence we now take $x\in\mathit{\Lambda}$ and fix $r=(1/2)diam\mathit{\Lambda}$. If $y\in\mathit{\Lambda}$ is sufficiently close to $x$ then there exist $t$ sufficiently close to $s$ such that $\Gamma_Z(t,y)=\Gamma_Z(s,x)=(-1,0)$. Then there exist positive global trajectories $\Gamma_x^+$ and $\Gamma_y^+$ passing through $x$ and $y$, respectively, satisfying $d(\Gamma_x^+(T),\Gamma_y^+(T))>r$ for some $T>0$ statement $(3)$ is proved.
	Finally, since every arbitrary point of $\Lambda$ can be connect to itself by some orbit we can perform a suitable concatenation $\widetilde{\Gamma}$ to conclude the last part of the Proposition.
\end{proof}


\section{Limit sets of some discontinuous linear vector fields}\label{linear}



Discontinuous linear vector fields have been widely studied due to its theoretical and practical importance. Moreover, this class of Filippov systems is particular interesting because one can easily find its solutions.
Regardless, in what follows we see that linear DVF not only verify the conditions of Theorem \ref{TPB_Euzebio-Juca} but the $\omega$-limit set of a maximal trajectory may be a chaotic set of type III. 

Let $Z=(X,Y)$ be a planar discontinuous linear vector field with $\Sigma=\{(0,y):\ y\in\R\}$ and 
$X_\pm(x,y)=A^\pm(x,y)^T+b^\pm$, being
$$
A^\pm = \left(\begin{array}{ccc}
a_{11}^\pm & a_{12}^\pm\\
a_{21}^\pm & a_{22}^\pm
\end{array}
\right) \quad \mbox{and} \quad b^\pm = (b_1^\pm,b_2^\pm)^T.
$$

Assume that $det(A^\pm)\neq 0$ so $X$ or $Y$ have only one equilibrium which we assume to be located outside $\Sigma$. Note that these are generic assumptions.
Also, notice that for $p=(0,y)\in\Sigma$ we have
$X_\pm. f(p)= a_{12}^\pm y+b_1^\pm$, and $X_\pm^2. f(p) = a_{12}^\pm(a_{22}^\pm y+b_2^\pm)$.
Consequently there exist at most one tangency point for $X$ or $Y$ at $p^\pm=\left(0,-b_1^\pm/a_{12}^\pm\right)$ 
if $a_{12}^\pm\neq 0$. It is easy to check that $p^\pm$ is either a fold point of $X$ or $Y$ or a tangency point of infinite order. Moreover the pseudo-equilibria are isolated. Therefore $Z$ satisfies the hypotheses of Theorem \ref{TPB_Euzebio-Juca} so the $\omega$-limit set of any maximal trajectory $\Gamma_Z(t,p)$ contained on a compact set is one of the possibilities from (i) to (ix) of Theorem \ref{TPB_Euzebio-Juca} excepted by (iii).

%

We can also provide a more detailed version of Proposition \ref{cor_chaos} in terms of the parameters of the system. Indeed, with the notations introduced previously we obtain the following result.

\begin{proposition}
	If a linear DVF $Z=(X,Y)$ has chaotic behavior then the following statements hold (simultaneously):
	\begin{itemize}
		\item[(i)] $Y$ and $X$ have coincident tangency points, that is, $a_{12}^-b_1^+-a_{12}^+b_1^-=0$;
		\item[(ii)] the double tangency is parabolic or hyperbolic, that is, $a_{12}^-b_2^--a_{22}^-b_1^-<0$ or $a_{12}^+b_2^+-a_{22}^+b_1^+>0$, respectively;
		\item[(iii)] $\s^c=\emptyset$, that is, $a_{12}^+a_{12}^-<0$.
	\end{itemize}
\end{proposition}






To exemplify the richness of discontinuous linear vector fields we now consider $Z=(X,Y)$ separated by the straight line $\Sigma=\{(0,y):y\in\R\}$, where 
$$
X(x,y)=\left(-\frac{1}{2}x-y-1,x+\frac{1}{2}y+2\right)\ \ \mbox{and}
$$
$$
Y(x,y)=(x+y+1,-2x-y-2).
$$

The equilibrium points are centers located at $(-2,0)$ for $X$ and at $(-1,0)$ for $Y$. Setting $\Sigma=f^{-1}(0)$ with $f(x,y)=x$ we get $Xf(0,y)=-y-1$ and $Y(0,y)=y+1$. So both vector fields 
$X$ and $Y$ have tangency points at $(0,-1)$ which is visible for $X$ and invisible for $Y$. Since $Xf(0,1)<0$ and 
$Yf(0,1)>0$ it follows that $\Sigma^e=(-1,+\infty)$ and $\Sigma^s=(-\infty,-1)$. A simple calculation shows that the Filippov vector field is given by $\dot{x}=0, \dot{y}=-y/4$. Therefore $(0,0)$ is the unique pseudo-equilibrium point. Moreover, it can be extended beyond the tangency point $(0,-1)$ which in this case becomes an attractor equilibrium point for the extended Filippov system (see Figure \ref{fig1-10}). Now we construct the chaotic set of type III as follow: let $\widetilde{\Lambda}\subset\R^2$ be the subset delimited by the following curves
\begin{itemize}
	\item[$\g_1:$] the orbit of $X$ connecting the pseudo-equilibrium $(0,0)$ to $\s^s$;
	\item[$\g_2:$] the orbit of $Y$ connecting $(0,0)$ to $\s^s$;
	\item[$\g_3:$] the periodic orbit passing through the tangency point $(0,-1)$ according to $X$.
\end{itemize}

\begin{center} 
	\includegraphics[scale=1.4]{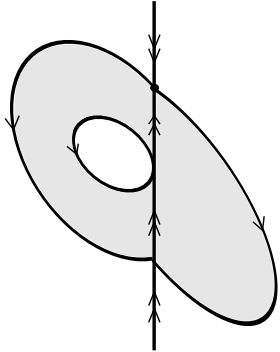}
	\captionof{figure}{A positively minimal set realized by a linear DVF.}\label{fig1-10}
\end{center}

That lead us to the following result.

\begin{proposition}
	$\widetilde{\Lambda}$ is a chaotic set of type III and it is a positively minimal set.
\end{proposition}

\begin{proof}
	The fact that $\widetilde{\Lambda}$ is a chaotic set of type III follows from step 4 in the proof of the Lemma \ref{aff6}. Besides, $\widetilde{\Lambda}$ is clearly a non empty, compact and positively invariant set. In order to see that $\widetilde{\Lambda}$ is indeed minimal, we notice that if $p\in int(\Lambda)$ then no positive maximal trajectory $\g^+_Z(t,p)$ reaches the pseudo-equilibrium $(0,0)$ and therefore $\g_Z^+(t_0,p)$ is contained on the interior of $\Lambda$. On the other hand, if $q\in\partial\Lambda$ then the saturation of $q$ in future time is the whole $\Lambda$ except possibly by a segment on its boundary. Thus $\widetilde{\Lambda}$ cannot contain any non empty, compact and positively invariant proper subset, that is, $\widetilde{\Lambda}$ is a positively minimal set.
\end{proof}



\section{Conclusions}\label{final}

\medskip

In this paper we studied some objects emerging from the theory of DVF as $\omega$-limit of a maximal trajectory $\g_Z(t,p)$ on two-dimensional Riemannian manifolds. In particular we introduce new objects which actually are a sophistication of some known concepts. More precisely, we distinguish pseudo-cycles from mild pseudo-cycles and we study three different types of chaotic sets. Such a refinement is necessary once structural unstable situations may occur as $\omega$-limit even in simple contexts as the linear one. In par\-ti\-cu\-lar for the linear case, we provide some classes for which the main results of the paper applies.


We also observed the absence of nontrivial minimal sets under the hypotheses of Theorem \ref{TPB_Euzebio-Juca} because generally invariance cannot be guaranteed unless we assume (i) more tangency points (consequently more sliding and escaping regions) or (ii) more regions defining the DVF. Besides that, we identified the existence of {\it orientable} minimality and chaos.




It is important to highlight that the global study performed throughout this paper allow us to identify not only $\omega-$limit sets but also the own structure of trajectories in a more embracing scenario. For instance, besides the facts concerning minimal sets, chaos and transitivity, we detect the existence of homoclinic and heteroclinic connections, several different types of limit sets (crossing, tangential or having sliding) as well new objects as tangency points of type II where trajectories may also accumulate.

Finally we call the attention to a potential application of the global analyzes performed in this paper addressing bifurcation theory. Indeed, let $Z=(X,Y)$ be a DVF such that each vector fields $X$ and $Y$ contribute with one tangency point of even order, $p^-$ and $p^+$ with $p^-\neq p^+$. Assume that these tangency points are connected through a maximal trajectory. This trajectory is a tangential pseudo cycle having two tangent points on it and a pseudo-equilibrium of saddle type inside (see Figure \ref{fig1-9}). One can study the bifurcation of sliding limit cycles of different topological types by breaking the fold connections. Indeed, the limit cycles may occupy one or two zones as well as they can be formed by one or two arcs of sliding. The same occurs in other configurations studied in the paper. To understand such bifurcations and to provide their unfolding is a hard task that have been under explored in the literature despite the exhaustive number of phenomena modeled by DVF.

%


%

\begin{center}
	\includegraphics[scale=0.85]{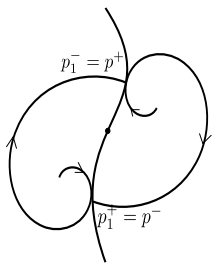}
	\captionof{figure}{Two fold visible-visible connection.}\label{fig1-9}			
\end{center}

\vspace{.5cm}

\noindent {\textbf{Acknowledgments.} This document is the result of the research
	projects funded by Pronex/ FAPEG/CNPq  grant 2012 10 26 7000 803 and grant 2017 10 26 7000 508 (Euzébio), Capes grant 88881.068462/2014-01 (Euzébio and Jucá), Universal/CNPq grant 420858/2016-4 (Euzébio) and CAPES Programa de Demanda Social - DS (Jucá).}





\end{multicols}

\end{document}